\documentclass{article}

\usepackage{amsmath,amssymb,amsthm}

\usepackage[T1]{fontenc}
\usepackage{libertine}
\usepackage[libertine]{newtxmath}
\usepackage{tikz-cd}
\usepackage{tikz}
\usepackage{mathrsfs}
\usepackage{amstext}
\usepackage{amsmath}
\usepackage{amscd}
\usepackage[mathscr]{euscript}
\usepackage{xspace}
\usepackage{bbm}

\tikzcdset{arrow style=math font}

\numberwithin{equation}{section}
\newtheorem{thm}{Theorem}[section]
\newtheorem{cor}[thm]{Corollary}
\newtheorem{prop}[thm]{Proposition}
\newtheorem{lem}[thm]{Lemma}

\theoremstyle{definition}
\newtheorem{defn}[thm]{Definition}
\theoremstyle{remark}
\newtheorem{rmk}[thm]{Remark}
\newtheorem{exam}[thm]{Example}

\newtheorem{obs}[thm]{Observation}
\usepackage{microtype}

\usepackage[all]{xy}

\usepackage{todonotes}

\usepackage{hyperref}
\hypersetup{
  colorlinks   = true, 
  urlcolor     = blue, 
  linkcolor    = blue, 
  citecolor   = red 
}

\newcommand{\category}{{$\infty$-cat\-e\-go\-ry}\xspace}
\newcommand{\categories}{$\infty$-cat\-e\-gories\xspace}
\newcommand{\categorical}{{$\infty$-cat\-e\-gor\-i\-cal}\xspace}

\newcommand{\Cech}{$\check{\text{C}}$ech\xspace}
\newcommand{\etale}{\'{e}tale\xspace}

\newcommand{\op}{{op}}
\newcommand{\act}{{act}}

\newcommand{\zz}{\mathbb{Z}}

\newcommand{\nn}{\mathbb{N}}

\newcommand{\sphere}{\mathbb{S}}

\newcommand{\brak}[1]{\langle #1 \rangle}

\newcommand{\msc}[1]{\mathscr{#1}}
\newcommand{\unit}{\mathbbm{1}}

\DeclareMathOperator{\Map}{Map}

\DeclareMathOperator{\Fun}{Fun}

\DeclareMathOperator{\Sp}{Sp}
\DeclareMathOperator{\Fin}{Fin}

\DeclareMathOperator{\Env}{Env}
\DeclareMathOperator{\Alg}{Alg}

\DeclareMathOperator{\CAlg}{CAlg}
\DeclareMathOperator{\CAlgTow}{CAlg^{Tow}}

\DeclareMathOperator{\CycSp}{CycSp}

\DeclareMathOperator{\ev}{ev}
\DeclareMathOperator{\Act}{act}

\DeclareMathOperator{\THH}{THH}

\DeclareMathOperator{\tow}{Tow}
\DeclareMathOperator{\THHDay}{THH^{Tow}}
\DeclareMathOperator{\TC}{TC}

\DeclareMathOperator{\Fil}{Fil}

\DeclareMathOperator{\gr}{gr}
\DeclareMathOperator{\May}{May}
\DeclareMathOperator{\Assoc}{Assoc}

\DeclareMathOperator*{\Day}{\circledast}

\usepackage[top=1in, left=1.5in, right=1.5in, bottom=1in]{geometry}

\title{The May filtration on THH and faithfully flat descent}
\author{Liam Keenan}

\begin{document}
\maketitle

\begin{abstract}
	In this paper, we prove that both topological Hochschild homology and topological cyclic homology are sheaves for the fpqc topology on connective commutative ring spectra, by exploiting the May filtration on topological Hochschild homology.  	
\end{abstract}

\tableofcontents

\section{Introduction}

Topological Hochschild homology ($\THH$) and topological cyclic homology ($\TC$), along with their Dennis and cyclotomic trace maps, have proved to be very useful tools for calculating algebraic $K$-theory. Algebraic K-theory has the nice property of being a sheaf for the Nisnevich topology; in practice, this allows one to decompose calculations into more manageable pieces. However, K-theory fails to satisfy the more computationally useful \etale descent. Fortunately, $\THH$ and $\TC$ do not have this deficiency; they have fpqc descent for commutative rings, by a theorem of Bhatt-Morrow-Scholze in \cite{BMS2}. Additionally, they have \etale descent for  $\mathbb{E}_{\infty}$-ring spectra by a theorem of Clausen-Mathew in \cite{clausen-mathew}. At this stage, it is reasonable to wonder whether it is possible to promote fpqc descent to ring spectra. Our work answers this in the affirmative, provided we work with connective commutative ring spectra, and our main theorem is the following: 

\begin{thm}\label{thm:maintheorem}
	Let $F = \THH$, $\THH(-)^{h\mathbb{T}}$, $\THH(-)^{hC_{p}}$, $\THH(-)_{hC_{p}}$, $\THH(-)^{tC_{p}}$, or $\TC$ viewed as functors $\CAlg_{\geq 0} \rightarrow \Sp$. Then $F$ satisfies descent for the fpqc topology on $\CAlg_{\geq 0}$. 	
\end{thm}

One would hope that our proof might follow the same strategy as for commutative rings, but working with spectra excludes this possibility. In \cite{BMS2}, $\THH$ is shown to have fpqc descent by a series of reductions. First, the authors reduce the problem to the case of Hochschild homology, by using the Postnikov tower of $\sphere$, and then reduce to the case of wedge powers of the cotangent complex, via the Hochschild-Konstant-Rosenberg (HKR) filtration on Hochschild homology. Finally, the case of the cotangent complex is handled by a theorem of Bhatt, a proof of which appears as \cite[3.1]{BMS2}. Since many of the tools used in \cite{BMS2} are unique to working with (simplicial) commutative rings, different methods became necessary. Our strategy exploits structural properties of the May filtration on $\THH$, as introduced in \cite{angelini-knoll-salch}. In particular, this filtration is Hausdorff, and its associated graded can be calculated via a canonical equivalence $\gr_{\ast}^{\May}\THH(A) \simeq \THH(H\pi_{\ast}A)$. By proving Theorem \ref{thm:maintheorem} in the case of generalized Eilenberg-Mac Lane spectra, we bootstrap to the general case using our control of the May filtration. 

We now provide a summary of our work. In Section 2, we carefully prove several folklore results on filtered spectra, all of which are undoubtedly known to experts. In particular, we prove that the "evaluate at zero" functor and the associated graded functor $\Fun(\mathbb{Z}^{\op}_{\geq 0},\Sp) \rightarrow \Sp$ are symmetric monoidal for the Day convolution, and that the Whitehead tower functor $\Sp \rightarrow \Fun(\zz^{\op}_{\geq0},\Sp)$ is also symmetric monoidal. In Section 3, we review some facts about \Cech conerves and prove a special case of our main theorem for generalized Eilenberg-Mac Lane spectra. In Section 4, we present an \categorical treatment of the May filtration on $\THH$ and use this to construct a similar filtration on the $C_{p}$-homotopy fixed points of $\THH$. Before doing so, we give a brief review of the symmetric monoidal envelope construction which appears in \cite[2.2.4]{HA}, which is crucial to the \categorical treatment of the May filtration. In Section 5, we turn to the proof of our main theorem, and use the May filtration in conjunction with our work in Section 3 to deduce the general result. 

Throughout this work, we freely use the theory of \categories, incarnated via quasicategories, as in \cite{HTT} and \cite{HA}. For consistency, wherever possible, we also follow the notation therein. Throughout, we refer to the \category of spectra by $\Sp$ and use $\otimes$ to denote the smash product. Additionally, we will use the notation $\Sp_{\geq 0}$ to denote connective spectra, and $\CAlg_{\geq 0}$ to denote connective commutative ring spectra. We will fix the following convention that a spectrum $X$ is $n$-connective provided that $\pi_{m}X \cong 0$ for $m < n$, and a map $f: X \rightarrow Y$ of spectra is $n$-connective provided that $\pi_{m}f$ is an isomorphism in degrees $m < n$ and a surjection when $n = m$. We will make frequent use of the groundbreaking work of Nikolaus-Scholze, providing an \categorical treatment of topological Hochschild homology, topological cyclic homology, and cyclotomic spectra, \cite[III.2.3, II.1.8, II.1.6]{nikolaus-scholze}. 

\textbf{Acknowledgements:} The author extends his greatest thanks to his Ph.D. advisor, Tyler Lawson. His insights, suggestions, and patience have been invaluable throughout the course of this work. The author would also like to thank John Rognes, for his suggestion to use the work of \cite{dundas-rognes}, and Martin Speirs and Joel Stapleton, for providing helpful feedback on a draft. Additionally, conversations with Micah Darrell and David Mehrle helped improve the author's understanding of this material.

\section{Filtered spectra}

Our main technical tool is the May filtration on $\THH$, as defined in \cite{angelini-knoll-salch}, which we must import into the \categorical context. To do so, we establish several folklore results about filtered (resp. $\mathbb{N}$-filtered) spectra. While the May filtration we consider is indexed by $\zz^{op}_{\geq 0}$, it will be technically useful to work with $\zz^{op}$-indexed filtrations as well. Unless otherwise specified, $\zz^{op}$ and $\zz^{op}_{\geq 0}$ carry the usual ordering, and by abuse, we use the same notation to refer to the nerves of these posets. 

\begin{defn}
	The \category of filtered spectra is $\Fil(\Sp) = \Fun(\mathbb{Z}^{op},\Sp)$, and the \category of $\mathbb{N}$-filtered spectra is $\tow(\Sp) = \Fun(\zz^{op}_{\geq 0},\Sp)$. 
\end{defn}

It is readily checked that both of these \categories are presentably symmetric monoidal and stable where the symmetric monoidal product in both cases is Day convolution; see \cite[5.5.3.6]{HTT}, \cite[1.1.3.1]{HA}, and \cite[2.13]{glasman2013day}. Throughout, we will let $\Day$ denote the Day convolution product and allow context to dictate exactly which product we mean. The unit object in $\Fil(\Sp)$, denoted by $\unit_{\Fil}$, is the filtered spectrum 
$$ \cdots \rightarrow 0 \rightarrow 0 \rightarrow \mathbb{S} \rightarrow \mathbb{S} \rightarrow \cdots $$
which is $\mathbb{S}$ in degrees $n \leq 0$ with identity maps in negative degree. The unit object in $\tow(\Sp)$, denoted by $\unit_{\tow}$, is the $\nn$-filtered spectrum which is $\mathbb{S}$ in degree $0$, and $0$ otherwise:
$$\cdots \rightarrow 0 \rightarrow 0 \rightarrow \mathbb{S}$$
Because $\Fil(\Sp)$ and $\tow(\Sp)$ are symmetric monoidal, we may consider the \categories of associative (resp. commutative) algebra objects, which we denote by $\Alg^{\Fil}$ (resp. $\CAlg^{\Fil}$) and $\Alg^{\tow}$ (resp. $\CAlgTow$). These categories will appear later when we work with the May filtration on $\THH$. The following are the technical claims we seek to prove: 

\begin{enumerate}
	\item The evaluation functor $\ev_{0} : \tow(\Sp) \rightarrow \Sp$ is symmetric monoidal;
	\item the associated graded functor $\gr_{\ast} : \tow(\Sp) \rightarrow \Sp$ is colimit-preserving and symmetric monoidal functor; and
	\item the Whitehead tower functor $\Sp \rightarrow \tow(\Sp)$ is lax monoidal. 	
\end{enumerate}

Before commencing, we will record several general results regarding $\Fil(\Sp)$, $\tow(\Sp)$, and their relationship to one another. 

\subsection{Properties of $\Fil(\Sp)$, $\tow(\Sp)$, and Day convolution}

Recall, that for $C$ a stable $\infty$-category, an object $X \in C$ is said to \emph{generate} $C$, provided that $\pi_{0}\Map_{C}(X,Y) \simeq \ast$ implies that $Y \simeq 0$. We will say that a collection of objects of $C$, $\{X_{i}\}_{i \in I}$, \emph{jointly generates} $C$, provided that $\pi_{0}\Map(X_{i},Y) \simeq \ast$ for all $i \in I$ implies that $Y \simeq 0$. 

\begin{lem}
	Let $K$ be a simplicial set and let $\ev_{k} : \Fun(K,\Sp) \rightarrow \Sp$ denote evaluation at $k \in K$. For all objects $k \in K$, the functors $\ev_{k}$ admit left adjoints, $L_{k} : \Sp \rightarrow \Fun(K,\Sp)$ given on vertices by $L_{k}X : k' \mapsto X^{\otimes \Map_{K}(k,k')}$. Additionally, the objects $S(k) = L_{k}\mathbb{S}$ have the following properties:	
	\begin{enumerate}
	\item $S(k)$ is compact for all $k \in K$;
	\item the $S(k)$'s jointly generate $\Fun(K,\Sp)$; and
	\item the collection $\{\Sigma^{n}S(k)\}_{n \in \zz, k \in K}$ generates $\Fun(K,\Sp)$ under small colimits. 		
	\end{enumerate}
\end{lem}

\begin{proof}
	Since $\ev_{k}$ preserves small limits and colimits, it admits a left adjoint $L_{k} : \Sp \rightarrow \Fun(K,\Sp)$. The compactness of $S(k)$ follows from the fact that $\mathbb{S}$ is compact in $\Sp$ and $\ev_{k}$ preserves colimits. To explicitly identify $L_{k}$, we use the following chain of natural equivalences obtained from the end formula, which appears for example, in \cite[5.2]{laxcolimits}
	\begin{align*}
	\Map_{\Fun(K,\Sp)}(X^{\otimes \Map_{K}(k,-)},E_{\bullet}) & \simeq \underset{i \to j \in \text{Tw}(K)}{\varprojlim}\Map_{\Sp}\big(\underset{\Map_{K}(k,i)}{\varinjlim}  X, E_{j}\big) \\
	& \simeq \underset{i \to j \in \text{Tw}(K)}{\varprojlim} \  \underset{\Map_{K}(k,i)}{\varprojlim}	\Map_{\Sp}(X,E_{j}) \\
	& \simeq \Map_{\Sp}(X, \underset{i \to j \in \text{Tw}(K)}{\varprojlim} E_{j}^{\Map_{K}(k,i)})
	\end{align*}
	As the exponential object $E_{j}^{\Map_{K}(k,i)}$ is equivalent to the mapping spectrum $\underline{\Map}_{\Sp}(\Sigma^{\infty}_{+}\Map_{K}(k,i),E_{j})$, an application of the end formula and the spectral co-Yoneda lemma yield: 	
	\begin{align*}\Map_{\Sp}(X, \underset{i \to j \in \text{Tw}(K)}{\varprojlim} E_{j}^{\Map_{K}(k,i)})  & \simeq \Map_{\Sp}\left(X, \underline{\Map}_{\Fun(K,\Sp)}(\Sigma_{+}^{\infty}\Map_{K}(k,-), E_{\bullet}) \right) \\
	& \simeq \Map_{\Sp}(X,E_{k}).
	 \end{align*} 
	This proves that $L_{k}$ is given on vertices as claimed.	
	\par To prove the final assertions, note that as $\Map_{\Fun(K,\Sp)}(S(k),E_{\bullet}) \simeq E_{k}$, the collection $\{S(k)\}_{k \in K}$ jointly generates the stable $\infty$-category $\Fun(K,\Sp)$. This allows us to mimic the proof of \cite[1.4.4.2]{HA} to show that the objects $\Sigma^{n}S(k)$ generate $\Fun(K,\Sp)$ under small colimits.	
\end{proof}

\begin{exam}\label{exam:T(n)}
	In the case where $K = \zz^{op}$, the object $S(k)$ is the tower with $\mathbb{S}$ in degrees $\leq k$ and $0$ otherwise. The maps $S(k)_{n}  \rightarrow S(k)_{n-1}$ are the identity if both objects are $\mathbb{S}$. Similarly, if $K = \mathbb{Z}^{op}_{\geq 0}$, the objects $S(k)$ are given by $\mathbb{S}$ with identity maps in the range $[0,k]$ and $0$ otherwise.
\end{exam}

This next lemma will give us a convenient way to calculate Day convolution products of ($\nn$-)filtered spectra.

\begin{lem}\label{lem:CalcDay}
	Let $\oplus : \zz^{op} \times \zz^{op} \rightarrow \zz^{op}$ denote the monoidal product and let $A_{k} \subseteq B_{k} = (\zz^{op} \times \zz^{op}) \times_{\zz^{op}} \zz^{op}_{/k}$ denote the full subcategory of those pairs $(n,m)$ such that $k + 1 \geq n+ m \geq k$. Then the inclusion $A_{k} \subseteq B_{k}$ is cofinal.	
\end{lem}

\begin{proof}
	This is a straightforward application of the \categorical version of Quillen's Theorem A, \cite[4.1.3.1]{HTT}. For $(r,s) \in B_{k}$, where $r+s > k$, the nerve of $A_{k} \times_{B_{k}} (B_{k})_{(r,s)/}$ is equivalent to finitely many copies of $\Lambda^{2}_{1}$ glued together along their initial and terminal vertices in a zig-zag pattern, hence contractible. If $r + s = k$, then the nerve of $A_{k} \times_{B_{k}} (B_{k})_{(r,s)/}$ is contractible. In either case, the nerve is contractible, so we may conclude by Quillen's Theorem A.	
\end{proof}

\begin{rmk}
	A virtually identical proof holds for $\zz^{op}_{\geq 0}$ as well. 	
\end{rmk}

Using Lemma \ref{lem:CalcDay}, we can now determine how the $L_{k}\sphere$'s interact through Day convolution. 

\begin{lem}\label{lem:T(n)}
	Let $K$ denote either $\zz^{op}$ or $\zz^{op}_{\geq 0}$. There is a natural equivalence 
	$$S(n) \Day S(m) \simeq S(n+m).$$	
\end{lem}

\begin{proof}
	This is a straightforward, though tedious, computation. By Lemma \ref{lem:CalcDay}, we can compute $\left(S(n) \Day S(m)\right)_{k}$ as the colimit of the diagram
	\begin{center}
		\begin{tikzcd}
		& S(n)_{s} \otimes S(m)_{k+1-s} \arrow[d] \arrow[dl] & S(n)_{s+1} \otimes S(m)_{k-s} \arrow[dl] \arrow[d] & \arrow[dl] S(n)_{s+2} \otimes S(m)_{k-s-1} \arrow[d] & \arrow[dl] \cdots \\
		\cdots & S(n)_{s} \otimes S(m)_{k-s} & S(n)_{s+1} \otimes S(m)_{k-s-1} & S(n)_{s+2} \otimes S(m)_{k-s-2} \\	
		\end{tikzcd}		
	\end{center}
	\noindent The objects are either of the form $S(n)_{s} \otimes S(m)_{k-s}$ or $S(n)_{s} \otimes S(m)_{k+1-s}$. When $k > n+m$, both $S(n)_{s} \otimes S(m)_{k-s}$ and $S(n)_{s} \otimes S(m)_{k+1-s}$ are trivial, and hence the colimit of the diagram above is $0$. In the case where $k \leq n + m$, the diagram is nonzero in a range: between the objects $S(n)_{k-m} \otimes S(m)_{k-(k-m)}$ and $S(n)_{n} \otimes S(m)_{k-n}$. In fact, in this range, all objects are given by $\mathbb{S}$ and all maps between nonzero objects are equivalent to $id : \mathbb{S} \rightarrow \mathbb{S}$, by Example \ref{exam:S(n)}. This implies that the colimit of the relevant diagram is given by $\mathbb{S}$. In summary, $S(n) \Day S(m)$ is zero in degrees $k > n+m$ and and given by $\sphere$ with identity maps in degrees $k \leq n+m$. 
\end{proof}

\begin{prop}\label{prop:FilTow}
	Let $\msc{E} \subseteq \Fil(\Sp)$ denote the full subcategory of those filtered spectra with the property that the maps $X_{n} \rightarrow X_{n - 1}$ are equivalences for $n \leq 0$. Then, $\msc{E}$ is a symmetric monoidal subcategory of $\Fil(\Sp)$, and the restriction map $\Fil(\Sp) \rightarrow \tow(\Sp)$ induced by the inclusion $\zz^{\op}_{\geq 0} \subseteq \zz^{op}$ is a symmetric monoidal equivalence $\msc{E} \rightarrow \tow(\Sp)$. 
\end{prop}

\begin{proof}
	By \cite[4.3.2.15]{HTT}, there is an equivalence of $\infty$-categories $\theta: \msc{E} \rightarrow \tow(\Sp)$ given by precomposing with the inclusion $\mathbb{Z}^{op}_{\geq 0} \subseteq \mathbb{Z}^{op}$. We will prove that $\msc{E}$ is a symmetric monoidal subcategory of $\Fil(\Sp)$ and that $\theta$ is symmetric monoidal. 
	
	Note that $S(n)$ in $\msc{E}$ restricts to $S(n)$ in $\tow(\Sp)$. Since the restriction map $\theta: \msc{E} \rightarrow \tow(\Sp)$ is an equivalence and compatible with the formation of iterated suspensions, we find that the desuspensions of the objects $S(n)$, where $n \geq 0$, generate $\msc{E}$ under small colimits. Thus, to show that $\msc{E}$ is a symmetric monoidal subcategory, we can immediately reduce to showing that $\Sigma^{s}S(n) \Day \Sigma^{t}S(m) \in \msc{E}$ since $\Day$ commutes with colimits separately in each variable. Additionally, since $\Sigma^{r} : \Fil(\Sp) \rightarrow \Fil(\Sp)$ is an equivalence for all $r \in \zz$, the result follows from Lemma \ref{lem:T(n)}. 
	
	It remains to prove that the canonical map
	$$\theta_{E,F} : \theta(E_{\bullet}) \Day \theta(F_{\bullet}) \rightarrow \theta(E_{\bullet} \Day F_{\bullet})$$
	\noindent is an equivalence for all objects in $\msc{E}$. Fix $E_{\bullet}$ and let $\msc{X}$ denote the full stable subcategory of $\msc{E}$ consisting of those towers $F_{\bullet}$ for which $\theta_{E,F}$ is an equivalence. Since this category is closed with respect to colimits, we can reduce to checking the claim for objects of the form $S(n)$ where $n \geq 0$. However, the natural map
	$$\theta(S(n)) \Day \theta(S(m)) \rightarrow \theta(S(n) \Day S(m)) \simeq \theta(S(n+m))$$
	\noindent is an equivalence by Lemma \ref{lem:T(n)}. This implies that $\theta$ is symmetric monoidal and hence $\tow(\Sp) \rightarrow \msc{E} \subseteq \Fil(\Sp)$ is symmetric monoidal. 
\end{proof}

For the remainder of Section 2, we proceed with the proofs of the desired folklore claims.

\subsection{The evaluation functor}

It will be convenient to follow the notation appearing in \cite{glasman2013day}, the relevant parts of which we now recall.

\begin{defn}
	Let $C^{\otimes}$ and $D^{\otimes}$ be two symmetric monoidal \categories. Define $\overline{\Fun(C,D)^{\otimes}}$ to be the simplicial set over $\Fin_{\ast}$ defined by the universal property
	$$ \Fun_{\Fin_{\ast}}(K,\overline{\Fun(C,D)^{\otimes}}) = \Fun_{\Fin_{\ast}}(C_{k}^{\otimes},D^{\otimes})$$
	\noindent where $C_{k}^{\otimes}$ is the pullback of $C^{\otimes} \rightarrow \Fin_{\ast}$ along the given structure map $k: K \rightarrow \Fin_{\ast}$. This simplicial set is an \category by \cite[2.3]{glasman2013day}. 
\end{defn}

\begin{obs}\label{obs:FinDecomp}
	If $S \in \Fin_{\ast}$, we let $S^{o}$ denote $S \setminus \{\ast\}$. Let $C^{\otimes}$ be a symmetric monoidal \category and let $(\Fin_{\ast})^{\act}_{/S}$ denote the full subcategory of $(\Fin_{\ast})_{/S}$ spanned by the active morphisms to $S$. Then for any $S \in \Fin_{\ast}$ there is a canonical product decomposition
	$$(\Fin_{\ast})_{/S} \simeq \left(\prod_{s \in S^{o}} (\Fin_{\ast})^{\act}_{/\{s\}_{+}}\right) \times \Fin_{\ast}$$
	\noindent This implies that for any morphism $f : S \rightarrow T$ in $\Fin_{\ast}$, there is a canonical decomposition 
	$$C_{f}^{\otimes} \simeq \left( \prod_{\Delta^{1},t \in T^{o}} C^{\otimes}_{\mu_{f^{-1}(t)_{+}}}\right) \times C^{\otimes}_{\beta_{f^{-1}(\ast)}}.$$
	\noindent For $V \in \Fin_{\ast}$, the map $\mu_{V}$ is given by the active map $V \rightarrow \langle 1 \rangle$ if $V$ is nonempty, and the inclusion $\ast \rightarrow \langle 1 \rangle$ if $V$ is empty. The map $\beta_{V}$ denotes the unique map $V \rightarrow \ast$. One can check that this decomposition of $C_{f}^{\otimes}$ is compatible with the decompositions of $C^{\otimes}_{S}$ and $C^{\otimes}_{T}$.	
\end{obs}

\begin{defn}\cite[2.8]{glasman2013day}
	Let $C^{\otimes}$ and $D^{\otimes}$ be symmetric monoidal \categories so that the tensor product of $D^{\otimes}$ preserves colimits separately in each variable. The \emph{Day convolution symmetric monoidal \category} is the largest simplicial subset of $\overline{\Fun(C,D)^{\otimes}}$ whose vertices over $S$ correspond to functors $F : C^{\otimes}_{S} \rightarrow D^{\otimes}_{S}$ which lie in the essential image of the inclusion
	$$ \prod_{s \in S^{o}} \Fun(C,D) \rightarrow \Fun(C_{S}^{\otimes},D_{S}^{\otimes}) $$
	\noindent and whose edges over $f : S \rightarrow T$ correspond to functors $F : C_{f}^{\otimes} \rightarrow D_{f}^{\otimes}$ in the essential image of the inclusion
	$$\prod_{t \in T^{o}} \Fun_{\Delta^{1}}(C^{\otimes}_{\mu_{f^{-1}(t)_{+}}}, D^{\otimes}_{\mu_{f^{-1}(t)_{+}}}) \times \Fun_{\Delta^{1}}(C^{\otimes}_{\beta_{f^{-1}(\ast)}}, D^{\otimes}_{\beta_{f^{-1}(\ast)}}) \rightarrow \Fun_{\Delta^{1}}(C^{\otimes}_{f},D^{\otimes}_{f})$$
	
	\noindent By \cite[2.10]{glasman2013day}, this is a symmetric monoidal \category, and in the case where $C^{\otimes}$ also preserves colimits separately, Day convolution preserves colimits separately in each variable \cite[2.13]{glasman2013day}. 
\end{defn}

\noindent With our notation set, we are ready to prove that $\ev_{0} :  \Fun(\zz^{op}_{\geq0},\Sp) \rightarrow \Sp$ is symmetric monoidal.

\begin{prop}\label{prop:Daycofinal}
	Let $C^{\otimes}, D^{\otimes}$ and $E^{\otimes}$ be symmetric monoidal $\infty$-categories where the tensor product of $E^{\otimes}$ preserves colimits separately in each variable, and let $F^{\otimes} : C^{\otimes} \rightarrow D^{\otimes}$ be a symmetric monoidal functor. Then there is an induced "pullback" functor
	\[ F^{\ast} : \Fun(D,E)^{\otimes} \rightarrow \Fun(C,E)^{\otimes} \]
	given by precomposition with $F^{\otimes}$. The "pullback" functor $\ev_{0}^{\otimes}$, which is induced by $0 : \Delta^{0} \rightarrow \zz^{op}_{\geq 0}$, is symmetric monoidal. 
\end{prop}

\begin{proof}
	The universal property of $\overline{\Fun(D,E)^{\otimes}}$ from \cite[2.1]{glasman2013day}, immediately implies that $F^{\otimes}$ induces a functor of $\infty$-categories
	\[ \overline{F^{\ast}} : \overline{\Fun(D,E)^{\otimes} }\rightarrow \overline{\Fun(C,E)^{\otimes}}. \]
	
	\noindent It now suffices to check that the restriction of $\overline{F^{\ast}}$ to $\Fun(D,E)^{\otimes}$ factors through $\Fun(C,E)^{\otimes}$. 
	
	The vertices of $\Fun(D,E)^{\otimes}$ over $S \in \Fin_{\ast}$ are precisely those functors $D_{S}^{\otimes} \rightarrow E^{\otimes}_{S}$ in the essential image of the inclusion $\Fun(D,E)^{S} \rightarrow \Fun(D_{S}^{\otimes},E_{S}^{\otimes})$. Since $F^{\otimes}$ is symmetric monoidal, $F^{\otimes}_{S} \simeq F^{S}$, so the vertices of $\Fun(D,E)^{\otimes}$ are sent to the vertices of $\Fun(C,E)^{\otimes}$. Similarly, the edges are functors $F_{f}^{\otimes} : D^{\otimes}_{f} \rightarrow E_{f}^{\otimes}$ in the essential image of the map
	\[ \left( \prod_{t \in T^{o}} \Fun_{\Delta^{1}}(D_{\mu_{S_{t}^{+}}}^{\otimes}, E^{\otimes}_{\mu_{S_{t}^{+}}})\right) \times \Fun(D^{\otimes}_{\beta_{S_{\ast}}}, E^{\otimes}_{\beta_{S_{\ast}}}) \rightarrow \Fun_{\Delta^{1}}(D^{\otimes}_{f},E^{\otimes}_{f}),\]
	
	\noindent where $S_{t}^{+} = f^{-1}(\{t\})_{+}$ and $S_{\ast} = f^{-1}(\{\ast\})$. As before, 
	$F^{\otimes}_{f} \simeq \left(\prod_{\Delta^{1}, t \in T^{o}} F^{\otimes}_{\mu_{S_{t}^{+}}} \right) \times F^{\otimes}_{\beta_{S_{\ast}}}$, which means the edges in $\Fun(D,E)^{\otimes}$ are sent to the edges in $\Fun(C,E)^{\otimes}$. This implies we have a well-defined functor $F^{\ast}$ as claimed.

	Let $K^{\otimes}$ denote the nerve of the category of operators of $\zz_{\geq 0}^{op}$ (this construction appears in  \cite[2.0.0.2]{HA}, for example). The map $0 : \Delta^{0} \rightarrow \zz_{\geq0}^{\op}$ induces a functor $e: \Fin_{\ast} \rightarrow K^{\otimes}$ which induces a "pullback" $\ev_{0}^{\otimes} : \tow(E)^{\otimes} = \Fun(\mathbb{\zz}_{\geq 0}^{op},E)^{\otimes} \rightarrow \Fun(\Delta^{0},E)^{\otimes} \simeq E^{\otimes}$. To show that $\ev_{0}^{\otimes}$ is symmetric monoidal, we must show that it takes cocartesian edges to cocartesian edges. By \cite[2.10]{glasman2013day}, the cocartesian edges over $f: S \rightarrow T$ in $\tow(E)^{\otimes}$ are precisely those functors $F_{f} : K^{\otimes}_{f} \rightarrow E_{f}^{\otimes}$ whose decomposition components are $p$-left Kan extensions. As $\ev_{0}^{\otimes}(F_{f})$ is the composite
	\begin{center}
		\begin{tikzcd}	
		\Delta^{1} \cong \Fin_{\ast} \times_{\Fin_{\ast}} \Delta^{1} \arrow[r] & K^{\otimes} \times_{\Fin_{\ast}} \Delta^{1} \arrow[r,"F_{f}"] & E_{f}^{\otimes} \arrow[r] & E^{\otimes},
		\end{tikzcd}		
	\end{center} 
	\noindent by \cite[4.3.1.4]{HTT} and \cite[4.3.1.15]{HTT} it suffices to show that the diagram 
	
	\begin{center}
		\begin{tikzcd}
		\Delta^{0} \arrow[r,"e_{S}"] \arrow[d] & K^{\otimes}_{S} \arrow[d] \arrow[r,"F_{S}"] & E^{\otimes}_{f} \arrow[d,"p"] \\
		\Delta^{1} \arrow[r,"e_{f}"] & K^{\otimes}_{f} \arrow[r] \arrow[ur,"F_{f}"] & \Delta^{1}	
		\end{tikzcd}		
	\end{center}
	
	\noindent exhibits $\ev_{0}^{\otimes}(F_{f})$ as the $p$-colimit of $F_{S} \circ e_{S}$. By assumption, we may decompose $F_{f}$ into a product of active and inert pieces, all of which are relative left Kan extensions. Furthermore, since the formation of overcategories commutes with pullbacks, we only need show that $\ev_{0}^{\otimes}(F_{g})$ is a $p$-colimit, where $g$ is one of the active or inert parts of $f$, as in Observation \ref{obs:FinDecomp}. Therefore, the result will follow if we can prove the following claim -- Let $f :S \rightarrow T$ be a morphism in $\Fin_{\ast}$, and let $F_{f} : K^{\otimes}_{f} \rightarrow E^{\otimes}_{f}$ be $p$-left Kan extension of $F_{f}|_{K^{\otimes}_{S}}$, exhibited by the diagram:
	
	\begin{center}
		\begin{tikzcd}
		K^{\otimes}_{S} \arrow[d] \arrow[r] & E^{\otimes}_{f} \arrow[d,"p"] \\
		K^{\otimes}_{f} \arrow[r] \arrow[ur] & \Delta^{1} 
		\end{tikzcd}
	\end{center}
	\noindent Then, $F_{f} \circ e_{f} : \Delta^{1} \rightarrow E_{f}^{\otimes}$ is a $p$-left Kan extension of $F_{f}|_{K^{\otimes}_{S}} \circ e_{S} : \Delta^{0} \rightarrow E_{f}^{\otimes}$ along the source inclusion $\Delta^{0} \subseteq \Delta^{1}$. To prove this, note that the edge $\Fin_{\ast} \times_{\Fin_{\ast}} \Delta^{1} \rightarrow K^{\otimes}_{f}$ is a lift of $f : S \rightarrow T$, with source $(0,\dots,0) \in N(\zz_{\geq 0}^{op})^{S} \simeq K_{S}^{\otimes}$. The structure of $K^{\otimes}$ forces the target of $\tilde{f}$ to be $(0,\dots,0) \in N(\zz_{\geq 0}^{op})^T \simeq K^{\otimes}_{T}$. The functors $\Delta^{0}_{/0} \rightarrow (K^{\otimes}_{S})_{/\tilde{f}(0)}$ and $\Delta^{0}_{/1} \rightarrow (K^{\otimes}_{S})_{/\tilde{f}(1)}$ are cofinal by \cite[4.1.2.6]{HTT}, as everything in sight is a contractible Kan complex. This shows that all the relevant induced diagrams are $p$-colimits, so that $\Delta^{1} \rightarrow E_{f}^{\otimes}$ is a $p$-left Kan extension, and hence a cocartesian arrow.
\end{proof}

\subsection{The associated graded functor}\label{section:gr}

In this section, we establish that the associated graded functor $\gr_{\ast} : \tow(\Sp) \rightarrow \Sp$ is symmetric monoidal by exploiting the result for $\Fil(\Sp)$, established in \cite[3.2.1]{lurie-k-rotation}. Let $(\zz^{\op})^{\text{disc}}$ be $\zz^{\op}$ endowed with the discrete topology, and let $\text{Gr}(\Sp) = \Fun((\zz^{op})^{\text{disc}},\Sp)$ denote the symmetric monoidal \category of graded spectra equipped with the Day convolution product

\begin{prop}\emph{\cite[3.2.1]{lurie-k-rotation}}
	There exists a functor, $\gr :  \Fil(\Sp) \rightarrow \emph{Gr}(\Sp)$, given on vertices by $X_{\bullet} \mapsto (X_{i}/X_{i-1})_{ i \in \zz}$, which is symmetric monoidal for the Day convolution and colimit preserving.	
\end{prop}

\noindent As $\text{Gr}(\Sp) \simeq \prod_{i \in \zz} \Sp$, there is a colimit-preserving symmetric monoidal functor $\bigoplus : \text{Gr}(\Sp) \rightarrow \Sp$ given on vertices by $(X_{i})_{i \in \zz} \mapsto \bigoplus_{i \in \zz} X_{i}$. Now define $\gr_{\ast}^{\Fil} : \Fil(\Sp) \rightarrow \Sp$ as $\bigoplus \circ \gr$, and let $\gr_{\ast} : \tow(\Sp) \rightarrow \Sp$ denote the composition of $\gr_{\ast}^{\Fil}$ with the symmetric monoidal functor, $ \tow(\Sp) \rightarrow \Fil(\Sp)$, from Proposition \ref{prop:FilTow}.

\begin{prop}\label{prop:grmonoidal}
	The associated graded functor $\gr_\ast : \tow(\Sp) \rightarrow \Sp$, is colimit preserving and symmetric monoidal. 
\end{prop}

\begin{proof}
	As $\tow(\Sp) \rightarrow \Fil(\Sp)$ is symmetric monoidal, it is immediate that $\gr_{\ast}$ is symmetric monoidal.  The fact that $\gr_{\ast}$ is colimit-preserving follows from the fact that $\gr^{\Fil}_{\ast}$ is colimit-preserving and that the functor $\tow(\Sp) \rightarrow \Fil(\Sp)$ is given by left Kan extension. 
\end{proof}

\begin{cor}\label{cor:Alggrmonoidal}
	$\gr_{\ast}$ induces functors $\gr_{\ast} : \Alg^{\tow} \rightarrow \Alg$ and $\gr_{\ast} : \CAlgTow \rightarrow \CAlg$ which preserve sifted colimits. 	
\end{cor}

\begin{proof}
	As $\gr_{\ast}$ is symmetric monoidal, it induces maps functors on commutative and associative algebras by definition (\cite[2.1.3.1]{HA}). Additionally, \cite[3.2.3.1]{HA} guarantees that $\gr_{\ast}$ preserves sifted colimits of associative and commutative algebras. 
\end{proof}

\subsection{The Whitehead tower}

Recall, that for all $n \in \zz$, there are functors $\tau_{\geq n} : \Sp \rightarrow \Sp_{\geq n}$, which are right adjoint to the inclusions $i_{n} : \Sp_{\geq n} \subseteq \Sp$ \cite[1.2.1.7]{HA}. The composites $i_{n} \circ \tau_{\geq n} : \Sp \rightarrow \Sp$ are colocalization functors, and the inclusions $\cdots \subseteq \Sp_{\geq n} \subseteq \Sp_{\geq n-1} \subseteq \cdots \subseteq \Sp_{\geq 0} \subseteq \Sp$ induce a diagram, $\zz^{op}_{\geq 0} \rightarrow \Fun(\Sp,\Sp)$, given by
$$\cdots i_{n} \circ \tau_{\geq n} \rightarrow i_{n-1} \circ \tau_{\geq n-1} \rightarrow \cdots \rightarrow i_{0} \circ \tau_{\geq 0}.$$

\noindent Using this diagram, we can construct a functor $T : \tow(\Sp) \rightarrow \tow(\Sp)$ which is given on vertices by $\{X_{n}\}_{n\geq 0} \mapsto \{\tau_{\geq n}X_{n}\}_{n\geq 0}$. 

\begin{defn}
	The \emph{Whitehead tower functor} is the composite $W = T \circ \delta : \Sp \rightarrow \tow(\Sp)$, where $\delta$ is the constant tower functor. 
\end{defn}

\begin{prop}
	The Whitehead tower functor $\Sp \rightarrow \tow(\Sp)$ is lax symmetric monoidal. 
\end{prop}

\begin{proof}
	Let $\msc{E} \subseteq \tow(\Sp)$ denote the full subcategory of $\nn$-filtered spectra, $X_{\bullet}$, with the property that $X_{n} \simeq \tau_{\geq n}X_{n}$. Additionally, note that the essential image of $T$ is exactly $\msc{E}$; this yields a right adjoint $R : \tow(\Sp) \rightarrow \msc{E}$, which is a colocalization. It remains to check the hypotheses of \cite[2.2.1.1]{HA} are satisfied, and by \cite[2.2.1.2]{HA} it suffices to show that $\msc{E}$ contains the unit and is closed with respect to the Day convolution for $\tow(\Sp)$. Certainly $\unit_{\tow} \in \msc{E}$, and if $X_{\bullet}, Y_{\bullet} \in \msc{E}$, 
	$$(X_{\bullet} \Day Y_{\bullet})_{n} = \underset{p+q \geq n}{\varinjlim} X_{p} \otimes Y_{q} \in \Sp_{\geq n}$$
	since $\Sp_{\geq n}$ is stable under colimits in $\Sp$. We conclude that $\msc{E} \subseteq \tow(\Sp)$ is symmetric monoidal and $R : \tow(\Sp) \rightarrow \msc{E}$ is lax monoidal, meaning that $T$ is also lax monoidal. As the constant tower functor $\delta : \Sp \rightarrow \tow(\Sp)$ is symmetric monoidal, we may conclude the result.
\end{proof}

\begin{rmk}
	Since $W$ is lax monoidal, we obtain functors $\Alg \rightarrow \Alg^{\tow}$ and $\CAlg \rightarrow \CAlgTow$. In particular, for $A \in \CAlg_{\geq 0}$, $W$ produces a multiplicative tower
	$$ \cdots \rightarrow \tau_{\geq n}A \rightarrow \tau_{\geq n-1}A \rightarrow \cdots \rightarrow \tau_{\geq 0}A \simeq A,$$
	\noindent where the $k$-th graded term is given by $\Sigma^{k}H\pi_{k}A$. 
\end{rmk}

\section{Descent for generalized Eilenberg-Mac Lane spectra}

In this section, we prove a special case of our main theorem. Before doing so, we establish some facts about \Cech conerves and recall a few definitions. 

\begin{defn}
	Let $C$ be an $\infty$-category which admits finite colimits, and let $f : X \rightarrow Y$ be an arrow in $C$. Then the \emph{augmented \Cech conerve of} $f$, denoted by $C^{\bullet}(f)_{+} : N(\Delta_{+}) \rightarrow C$ is the left Kan extension of $f : \Delta^{1} \rightarrow C$ along the inclusion $N(\Delta^{\leq 0}_{+}) \subseteq N(\Delta_{+})$. The \emph{\Cech conerve}, denoted by $C^{\bullet}(f)$ is the restriction of $C^{\bullet}(f)_{+}$ to $N(\Delta)$. In what follows, we may abuse notation and conflate $C^{\bullet}(f)_{+}$ and $C^{\bullet}(f)$. 
\end{defn}

\begin{exam}
	Let $f: A \rightarrow B$ be a morphism in $\CAlg$. As the coproduct in $\CAlg$ is given by the relative smash product, the augmented \Cech conerve is simply the augmented cobar complex of $f$:
	\begin{center}
	\begin{tikzcd}
	A \arrow[r] & B \arrow[r,shift right=1] \arrow[r,shift left=1] & B \otimes_{A} B \arrow[r] \arrow[r,shift right=2] \arrow[r,shift left=2] & \cdots 	
	\end{tikzcd}
	\end{center}
	We will be primarily interested in this cobar complex when $f$ is faithfully flat.
\end{exam}

\begin{defn}
	Let $f : A \rightarrow B$ be a morphism in $\CAlg$. We say that $f$ is \emph{faithfully flat}, provided the following conditions hold:
	\begin{enumerate}
	\item The map $\pi_{0}f : \pi_{0}A \rightarrow \pi_{0}B$ is a faithfully flat map of commutative rings; and 
	\item the induced map $(\pi_{0}B)\otimes_{\pi_{0}A} (\pi_{\ast}A) \rightarrow \pi_{\ast}B$ is an isomorphism of graded rings. 		
	\end{enumerate}	
\end{defn}

As in the discrete case, faithfully flat morphisms determine a Grothendieck topology on both $\CAlg$ and $\CAlg_{\geq 0}$, see \cite[B.6.1.3, B.6.1.7]{SAG}.

\begin{lem}\label{cechequivalence}
	Let $C$ be an $\infty$-category which admits finite colimits and let $F: \Delta^{1} \times \Delta^{1} \rightarrow C$ be a commutative square, such that $\alpha = F|_{\{0\} \times \Delta^{1}}$ and $\beta = F|_{\{1\} \times \Delta^{1}}$ are equivalences in $C$, which we depict as
	\begin{center}
	\begin{tikzcd}
	X \arrow[r,"f"] \arrow[d,"\alpha"'] & Y \arrow[d,"\beta"] \\
	X' \arrow[r,"g"'] & Y'
	\end{tikzcd}
	\end{center}
	Then, the induced map $C^{\bullet}(f)_{+} \rightarrow C^{\bullet}(g)_{+}$ is an equivalence of augmented cosimplicial objects.
\end{lem}

\begin{proof}
	Since $C$ admits finite colimits we can take the left Kan extension of $F$ along the inclusion $\Delta_{+}^{\leq 0} \times \Delta^{1} \subseteq \Delta_{+} \times \Delta^{1}$, yielding a morphism of augmented cosimplicial objects $\widetilde{F} : \Delta_{+} \times \Delta^{1} \rightarrow C$. It suffices to check that $\widetilde{F}|_{\{n\}\times \Delta^{1}}$ is an equivalence for all $n \in \Delta_{+}$. Since $\widetilde{F}|_{\{n\}\times \Delta^{1}}$ is given by the map $Y\coprod_{X} \cdots \coprod_{X} Y \rightarrow Y'\coprod_{X'} \cdots \coprod_{X'} Y'$ induced by $\alpha$ and $\beta$, it is an equivalence.  	
\end{proof}

\begin{lem}\label{cechpushout}
	Let $C$ be an $\infty$-category which admits finite colimits and let $D : \Delta^{1} \times \Delta^{1} \rightarrow C$ be a coCartesian square, depicted as
	\begin{center}
	\begin{tikzcd}
	X \arrow[r,"f"] \arrow[d,"\phi"'] & Y \arrow[d] \\
	X' \arrow[r,"f'"'] & Y'
	\end{tikzcd}
	\end{center}
	\noindent Then, the natural map $C^{\bullet}(f) \coprod_{X} X' \rightarrow C^{\bullet}(f')$ is an equivalence of cosimplicial objects of $C$. 	
\end{lem}

\begin{proof}	
	First, we construct a pushout functor. Let $\msc{E}$ be the full subcategory of $\Fun(\Delta^{1} \times \Delta^{1},C)$ consisting of pushout diagrams. The restriction functor $\msc{E} \rightarrow \Fun(\Lambda^{2}_{0},C)$ is an acyclic fibration by \cite[4.3.2.15]{HTT}. However, since  $\Lambda_{0}^{2} \cong (\partial\Delta^{1})^{\triangleleft} \cong \Delta^{1} \coprod_{\Delta^{\{0\}}} \Delta^{1}$ we know that $\Fun(\Lambda^{2}_{0},C) \cong \Fun(\Delta^{1},C) \times_{{C}} \Fun(\Delta^{1},C)$. Pulling back along the inclusion of the vertex $\{\phi: X \rightarrow X'\}$, we obtain an acyclic fibration
	\[ q: \msc{E} \times_{\Fun(\Delta^{1},C)} \{\phi\} \rightarrow \Fun(\Delta^{1},C)\times_{C} \{X\}. \]
	
	\noindent By choosing a section of $q$ and restricting to the opposite edge, we obtain a functor	
	\[ (-) \coprod_{X} X' : \Fun(\Delta^{1},C)\times_{C} \{X\} \rightarrow \msc{E} \times_{\Fun(\Delta^{1},C)} \{\phi\} \rightarrow \Fun(\Delta^{1},C) \times_{C} \{X'\} \]
	
	\noindent which takes $X \rightarrow Y$ to $X' \rightarrow Y \coprod_{X}X'$ and preserves colimits on account that the section of $q$ is a left adjoint and colimits in functor categories are computed pointwise. By taking left Kan extensions, we obtain the following commutative diagram:
	
	\begin{center}
	\begin{tikzcd}
	\Delta^{\leq 0}_{+} \arrow[d] \arrow[rr,bend left,"f'"] \arrow[r,"f"] & \Fun(\Delta^{1},C)\times_{C} \{X\} \arrow[r] & \Fun(\Delta^{1},C)\times_{C} \{X'\} \\
	\Delta_{+} \arrow[ur,"C^{\bullet}(f)_{+}"'] \arrow[urr,bend right=10,"C^{\bullet}(f')_{+}"'] 
	\end{tikzcd}
	\end{center}
	
	\noindent By the universal property, we obtain a natural map $C^{\bullet}(f)_{+} \coprod_{X}X' \rightarrow C^{\bullet}(f')_{+}$. This is an equivalence, since left Kan extensions are calculated pointwise as colimits and $(-)\coprod_{X} X'$ preserves them. Finally, by projecting down to $C$, we obtain the desired equivalence of augmented cosimplicial objects of $C$.
\end{proof}

We now record two key results used in our proof of Proposition \ref{prop:SpecialCase}.

\begin{thm}\emph{\cite[Corollary 3.4, Remark 3.5]{BMS2}} \label{thm:BMS2}
	The functors $\THH(-)$ and $\THH(-)_{hC_{p}}$ are $\Sp$-valued sheaves for the fpqc topology on the category of commutative rings.	
\end{thm}

\begin{thm}\emph{\cite[Theorem 1.2]{dundas-rognes}} \label{thm:dundas-rognes}
	Let $A \rightarrow B$ be a 1-connected morphism in $\CAlg_{\geq 0}$, and let $F = \THH$ or $\THH(-)_{hC_{p}}$. Then, the natural map
	\begin{center}
	\begin{tikzcd}
	F(A) \arrow[r] & \varprojlim \bigg( F(B) \arrow[r,shift right=1] \arrow[r,shift left=1] & F(B \otimes_{A} B) \arrow[r] \arrow[r,shift right=2] \arrow[r,shift left=2] & \cdots 	\bigg)
	\end{tikzcd}
	\end{center}
	\noindent is an equivalence in $\Sp$. 	
\end{thm}

\begin{prop}\label{prop:SpecialCase}
	Let $f: A \rightarrow B$ be a faithfully flat map in $\CAlg_{\geq 0}$, which induces a faithfully flat map $H\pi_{\ast}A \rightarrow H\pi_{\ast}B$. The natural map 	
	\begin{center}
	\begin{tikzcd}
	\THH(H\pi_{\ast}A) \arrow[r] & \varprojlim \bigg(\THH(H\pi_{\ast}B) \arrow[r,shift right = 1] \arrow[r,shift left = 1] & \THH(H\pi_{\ast}B \otimes_{H\pi_{\ast}A} H\pi_{\ast}B) \arrow[r] \arrow[r,shift right=2] \arrow[r,shift left=2] & \cdots \bigg)
	\end{tikzcd}
	\end{center}
	\noindent is an equivalence in $\Sp$. 
\end{prop}

\begin{proof}
	For brevity, we will write $A_{\ast}$ (resp. $B_{\ast}$) for $H\pi_{\ast}A$ (resp. $H\pi_{\ast}B$) and $A_{0}$ (resp. $B_{0}$) for $H\pi_{0}A$ (resp. $H\pi_{0}B$). Since $f$ is faithfully flat, we find that the commutative square
	\begin{center}
	\begin{tikzcd}
	A_{\ast} \arrow[r,"f"] \arrow[d,"\phi_{A_{\ast}}"'] & B_{\ast} \arrow[d,"\phi_{B_{\ast}}"] \\
	A_{0} \arrow[r,"g"'] & B_{0}
	\end{tikzcd}
	\end{center}
	\noindent is a pushout in $\CAlg_{\geq 0}$ by a K\"{u}nneth spectral sequence calculation. Viewing this square as a functor $\Delta^{\leq 0}_{+} \times \Delta^{\leq 0}_{+} \rightarrow \CAlg_{\geq 0}$ we can take its left Kan extension along the inclusion $\Delta^{\leq 0}_{+} \times \Delta^{\leq 0}_{+} \subseteq \Delta_{+} \times \Delta_{+}$, which we denote by $X_{+}^{\bullet,\bullet} : \Delta_{+}\times \Delta_{+} \rightarrow \CAlg_{\geq 0}$. By \cite[4.3.2.8]{HTT} and the dual of \cite[6.1.2.11]{HTT}, we can realize this diagram as the degreewise \Cech conerves of either $C^{\bullet}(f)_{+}\rightarrow C^{\bullet}(g)_{+}$ or $C^{\bullet}(\phi_{A_{\ast}})_{+} \rightarrow C^{\bullet}(\phi_{B_{\ast}})_{+}$. In virtue of this, we may identify the $n$-th row in the bisimplical diagram, $X_{+}^{\bullet,n}$, as the augmented \Cech conerve of the map $A_{0} \otimes_{A_{\ast}} \cdots \otimes_{A_{\ast}}A_{0} \rightarrow B_{0} \otimes_{B_{\ast}} \cdots \otimes_{B_{\ast}} B_{0}$, and the $m$-th column, $X_{+}^{m,\bullet}$, as the augmented \Cech conerve of the map $B_{\ast} \otimes_{A_{\ast}} \cdots \otimes_{A_{\ast}} B_{\ast} \rightarrow B_{0} \otimes_{A_{0}} \cdots \otimes_{A_{0}} B_{0}$. Note that the diagram, $\THH \circ \ X_{+}^{\bullet, \bullet}$, induces the following commutative square
	
	\begin{equation}\label{diagram:specialcase}
	\begin{tikzcd}
	\THH(X_{+}^{-1,-1}) \arrow[r] \arrow[d] & \underset{m\in \Delta}{\varprojlim} \THH(X_{+}^{m,-1}) \arrow[d] \\
	\underset{n\in \Delta}{\varprojlim} \THH(X^{-1,n}_{+}) \arrow[r] & \underset{n \in \Delta}{\varprojlim} \underset{m \in \Delta}{\varprojlim}  \THH(X_{+}^{m,n})
	\end{tikzcd}
	\end{equation}
	
	\noindent Since $X_{+}^{\bullet,-1}$ is the augmented \Cech conerve of the map $A_{\ast} \rightarrow B_{\ast}$, the result will follow if we can prove all the other arrows in Diagram \ref{diagram:specialcase} are equivalences. As the maps $X_{+}^{m,-1} \rightarrow X_{+}^{m,0}$ are 1-connected for all $m \geq -1$, Theorem \ref{thm:dundas-rognes} implies that the maps
	$$\THH(X_{+}^{m,-1}) \rightarrow \underset{n\in \Delta}{\varprojlim} \THH(X^{m,n}_{+})$$
	
	\noindent are equivalences for all $m \geq -1$. Thus, the vertical maps in the commutative square are equivalences, by commuting the limits in the bottom right corner. It remains to show that
	$$\THH(X_{+}^{-1,n}) \rightarrow \underset{m \in \Delta}{\varprojlim} \THH(X_{+}^{m,n})$$
	
	\noindent is an equivalence for all $n \geq 0$, which we prove by induction. The case $n = 0$ is precisely Theorem \ref{thm:BMS2}, so we may assume the result holds for $n$. As noted above, the $n$-th row of $X_{+}^{\bullet,\bullet}$ is given by the augmented \Cech conerve of the map $A_{0} \otimes_{A_{\ast}} \cdots \otimes_{A_{\ast}}A_{0} \rightarrow B_{0} \otimes_{B_{\ast}} \cdots \otimes_{B_{\ast}} B_{0}$. However, this augmented cobar construction is equivalent to the augmented \Cech conerve of the map 
	$$id \otimes \cdots \otimes id \otimes f_{0} : A_{0} \otimes_{A_{\ast}} \cdots \otimes_{A_{\ast}}A_{0} \rightarrow A_{0} \otimes_{A_{\ast}}  \cdots \otimes_{A_{\ast}} A_{0} \otimes_{A_{\ast}} B_{0}.$$
	
	\noindent Since the following diagram is a retract in $\CAlg$, 
	\begin{center}
	\begin{tikzcd}
	A_{0} \arrow[d,"f_{0}"] \arrow[r] & A_{\ast} \arrow[d,"f_{\ast}"] \arrow[r] & A_{0} \arrow[d,"f_{0}"] \\
	B_{0} \arrow[r] & B_{\ast} \arrow[r] & B_{0}
	\end{tikzcd}		
	\end{center}

	\noindent we obtain a retract diagram of augmented \Cech conerves
	$$C^{\bullet}(id^{\otimes n} \otimes f_{0})_{+} \rightarrow C^{\bullet}(id^{\otimes(n-1)} \otimes f_{0})_{+} \rightarrow C^{\bullet}(id^{\otimes n} \otimes f_{0})_{+}$$
	
	\noindent and hence a retract diagram	
	$$\THH\left( C^{\bullet}(id^{\otimes n} \otimes f_{0})_{+}\right) \rightarrow \THH\left(C^{\bullet}(id^{\otimes n-1} \otimes f_{0})_{+}\right) \rightarrow \THH\left(C^{\bullet}(id^{\otimes n} \otimes f_{0})_{+}\right).$$
	
	\noindent By the inductive hypothesis, we may conclude that $\THH\left( C^{\bullet}(id^{\otimes n} \otimes f_{0})_{+}\right)$ is a limit diagram. This completes the proof. 	
\end{proof}

\begin{prop}
	Let $f: A \rightarrow B$ be a faithfully flat map in $\CAlg_{\geq 0}$. Then the natural map
	
	$$\THH(H\pi_{\ast}A)_{hC_{p}} \rightarrow \varprojlim\left(\THH(H\pi_{\ast}B^{\otimes_{H\pi_{\ast}A}\bullet})_{hC_{p}}\right)$$
	
	\noindent is a limit diagram. 
\end{prop}

\begin{proof}
	This argument is identical to the one in Proposition \ref{prop:SpecialCase} since we can apply Theorems \ref{thm:dundas-rognes} and \ref{thm:BMS2} for $\THH(-)_{hC_{p}}$. 
\end{proof}

\section{The May filtration on $\THH$}

Let $A$ be a connective algebra spectrum and let $A_{\bullet}$ be its Whitehead tower. The May filtration on $\THH(A)$, as defined in \cite{angelini-knoll-salch}, is the geometric realization of the following diagram;

\begin{center}
\begin{tikzcd}
A_{\bullet} & A_{\bullet} \Day A_{\bullet} \arrow[l,shift left = 1] \arrow[l, shift right = 1] &  A_{\bullet} \Day A_{\bullet} \Day A_{\bullet} \arrow[l] \arrow[l,shift left = 2] \arrow[l,shift right = 2] & \arrow[l,shift right = 1] \arrow[l,shift right =3 ] \arrow[l,shift left =3] \arrow[l,shift left = 1] \cdots 	
\end{tikzcd}	
\end{center}

\noindent i.e. the cyclic bar construction of $A_{\bullet}$ with respect to the Day convolution product on $\tow(\Sp)$. In order to make this construction precise in the \categorical language, we mimic the construction of $\THH$ in \cite{nikolaus-scholze}. One prerequisite therein, is the symmetric monoidal envelope of \cite[2.2.4.1]{HA}. After a brief discussion of the monoidal envelope, we present a \categorical description of the May filtration, and establish several useful properties.

\begin{rmk}
	The cyclic bar construction can be performed for any associative algebra in a presentably symmetric monoidal \category. Since we are primarily interested in the Whitehead towers of algebra spectra and the Day convolution product, we will not work in such general terms, though many of the results in this section will hold in greater generality. 
\end{rmk}

\subsection{The symmetric monoidal envelope}

\begin{defn}\cite[2.2.4.1]{HA}
	Let $p: C^{\otimes} \rightarrow \Fin_{\ast}$ be a symmetric monoidal \category. The \emph{monoidal envelope} of $C^{\otimes}$, is the fiber product
	$$ \Env(C)^{\otimes} = C^{\otimes} \times_{\Fun(\{0\},\Fin_\ast)} \text{Act}(\Fin_{\ast}),$$
	\noindent where $\text{Act}(\Fin_{\ast}) \subseteq \Fun(\Delta^{1},\Fin_{\ast})$ is the full subcategory spanned by the active morphisms. 
	
	By \cite[2.2.4.4]{HA}, evaluation at $\{1\} \subseteq \Delta^{1}$ induces a map $q: \Env(C)^{\otimes} \rightarrow \Fin_{\ast}$, exhibiting $\Env(C)^{\otimes}$ as a symmetric monoidal \category. If we let $\Env(C) = \Env(C)^{\otimes}_{\brak{1}}$ denote the fiber of $q$ over $\brak{1}$, by unwinding the definition, we may identify $\Env(C)$ with the subcategory $C^{\otimes}_{\Act} \subseteq C^{\otimes}$ spanned by objects and active morphisms between them. In other words, $C^{\otimes}_{\Act}$ is a symmetric monoidal \category, and its monoidal product, $\oplus : C^{\otimes}_{\Act} \times C^{\otimes}_{\Act} \rightarrow C^{\otimes}_{\Act}$, may be described as follows: for $X \in C^{\otimes}_{\brak{n}}$ and $Y \in C^{\otimes}_{\brak{m}}$ corresponding to sequences $(X_{i})_{i=1}^{n}$ and $(Y_{j})_{j=1}^{m}$, the product $X \oplus Y$ corresponds to the concatenation $(X_{i})_{i=1}^{n} \cup (Y_{j})_{j=1}^{m} \in C^{\otimes}_{\brak{n+m}}$. 
\end{defn}

There is a diagonal embedding $\Fin_{\ast} \rightarrow \text{Act}(\Fin_{\ast})$, and the pullback along this embedding induces a lax symmetric monoidal inclusion $i_{C}: C^{\otimes} \subseteq \Env(C)^{\otimes}$, \cite[2.2.4.4, 2.2.4.16]{HA}. The following proposition provides a useful characterization of symmetric monoidal functors out of $\Env(C)^{\otimes}$. 

\begin{prop}\emph{\cite[2.2.4.9]{HA}}\label{prop:LurieEnv}
	Let $C^{\otimes}$ and $D^{\otimes}$ be symmetric monoidal \categories. The inclusion $i_{C} : C^{\otimes} \subseteq \Env(C)^{\otimes}$ induces an equivalence of \categories
	$$\Fun^{\otimes}(\Env(C),D) \simeq \Alg_{C}(D).$$
	\noindent Here, $\Fun^{\otimes}(\Env(C),D)$ denotes the \category of symmetric monoidal functors $\Env(C)^{\otimes} \rightarrow D^{\otimes}$ and $\Alg_{C}(D)$ is the \category of lax symmetric monoidal functors $C^{\otimes} \rightarrow D^{\otimes}$. 
\end{prop}

Under this equivalence, the identity functor $C \rightarrow C$, which is lax monoidal, corresponds to a symmetric monoidal functor, $\otimes_{C} : \Env(C) = C^{\otimes}_{\Act} \rightarrow C$, given on vertices by $(X_{1},\dots, X_{n}) \mapsto X_{1} \otimes \cdots \otimes X_{n}$. In order to establish several key properties of the May filtration on $\THH$, we need a lemma on the interaction of $\otimes_{C}$ with symmetric monoidal functors. Suppose we are given a symmetric monoidal functor, $F^{\otimes} : C^{\otimes} \rightarrow D^{\otimes}$. Pulling this morphism back along $\text{Act}(\Fin_{\ast})$, produces a functor $\Env(F)^{\otimes} : \Env(C)^{\otimes} \rightarrow \Env(D)^{\otimes}$ which is also symmetric monoidal by \cite[2.2.4.15]{HA}. $\Env(F)^{\otimes}$ fits into the following diagram,
\begin{center}
\begin{tikzcd}
C^{\otimes} \arrow[r,"F^{\otimes}"] \arrow[d, "i_{C}"'] & D^{\otimes} \arrow[d,"i_{D}"] \\
\Env(C)^{\otimes} \arrow[r,"\Env(F)^{\otimes}"'] & \Env(D)^{\otimes}
\end{tikzcd}	
\end{center}

\noindent which is commutative in virtue of the fact that $i_{C}$ and $i_{D}$ are induced by the diagonal.

\begin{lem}\label{lem:Env}
	Let $F^{\otimes} : C^{\otimes} \rightarrow D^{\otimes}$ be a symmetric monoidal functor. There is an equivalence
	\[ \otimes_{D} \circ \Env(F)^{\otimes} \simeq F^{\otimes} \circ \otimes_{C}.\]
\end{lem}

\begin{proof}
	As the equivalence in Proposition \ref{prop:LurieEnv} is given by precomposing by the lax symmetric monoidal inclusion $i_{C} : C^{\otimes} \subseteq \Env(C)^{\otimes}$, we can check the desired functors are equivalent by showing that 
	$$\otimes_{D} \circ \Env(F)^{\otimes} \circ i_{C}  \simeq F^{\otimes} \circ \otimes_{C} \circ i_{C}.$$
	\noindent This is immediate since $\Env(F)^{\otimes} \circ i_{C} \simeq i_{D} \circ F^{\otimes}$, $\otimes_{C} \circ i_{C} \simeq id$, and $\otimes_{D} \circ i_{D} \simeq id$. 	
\end{proof}

\begin{rmk}
	In the case where $C = \tow(\Sp)$, we will denote $\otimes_{C}$ by $\Day$, and in the case where $C = \Sp$, we will denote $\otimes_{C}$ by $\otimes$. 	
\end{rmk}

\subsection{The May filtration on $\THH$} 
\begin{defn}
	Let $A_{\bullet} \in \Alg^{\tow}$, given by a section  $A_{\bullet}^{\otimes} : \Assoc^{\otimes} \rightarrow \tow(\Sp)^{\otimes}$. Let $\THHDay(A_{\bullet}) \in \tow(\Sp)^{B\mathbb{T}}$ denote the colimit of the diagram	
	\begin{center}
	\begin{tikzcd}
	N(\Lambda^{\op}) \arrow[r,"V^{o}"] & \Assoc^{\otimes}_{\Act} \arrow[r,"A_{\bullet}^{\otimes}"] & \tow(\Sp)^{\otimes}_{\Act} \arrow[r,"\Day"] & \tow(\Sp),	
	\end{tikzcd}		
	\end{center}
	\noindent where $V^{o}$ is the map appearing in \cite[B.1]{nikolaus-scholze}.
\end{defn}

\begin{rmk}\label{rmk:GeneralTHH}
	This is essentially the definition of $\THH$ as given in \cite{nikolaus-scholze}, the only salient difference being that we are using the Day convolution of towers to form the cyclic bar construction as opposed to the smash product. Additionally, we know $\THHDay(A_{\bullet}) \in \tow(\Sp)^{B\mathbb{T}}$ since the geometric realization of a cyclic object admits a $\mathbb{T}$-action, \cite[B.5]{nikolaus-scholze}. 
\end{rmk}

\begin{prop}
	The functor $\THHDay : \Alg^{\tow} \rightarrow \tow(\Sp)^{B\mathbb{T}}$ is symmetric monoidal.
\end{prop}

\begin{proof}
	This essentially follows from the fact that Day convolution preserves sifted colimits. The careful reader should consult \cite[3.2.4.3, 3.2.4.3]{HA} to explicitly construct the functor
	$$(\THHDay)^{\otimes} : (\Alg^{\tow})^{\otimes} \rightarrow (\tow(\Sp)^{B\mathbb{T}})^{\otimes}$$
	\noindent which exhibits $\THHDay$ as symmetric monoidal. 
\end{proof}

\begin{cor}
	$\THHDay$ induces a functor $\CAlgTow \rightarrow (\CAlgTow)^{B\mathbb{T}}$. 	
\end{cor}

\begin{proof}
	Since $\THHDay$ is symmetric monoidal, it induces a functor 
	$$\CAlg(\Alg^{\tow}) \rightarrow \CAlg(\tow(\Sp)^{B\mathbb{T}}).$$
	\noindent As $\CAlg(\Alg^{\tow}) \simeq \CAlg(\tow(\Sp))$ and $\CAlg(\tow(\Sp)^{B\mathbb{T}}) \simeq (\CAlg^{\tow})^{B\mathbb{T}}$, we conclude the result.	
\end{proof}

For connective commutative ring spectra, the May filtration on $\THH$ can be given an alternative construction, following \cite{angelini-knoll-salch}. Since $\CAlgTow$ is presentable, it is automatically tensored over spaces \cite[4.4.4]{HTT}; we denote this by $X \otimes A_{\bullet}$. Using the simplicial model, $\Delta^{1}/\partial \Delta^{1}$, for $S^{1}$, a standard argument shows we may identify $S^{1} \otimes A_{\bullet}$ in $\CAlgTow$ with the cyclic bar construction of $A_{\bullet}$, where $A_{\bullet}$ is the Whitehead tower of $A \in \CAlg_{\geq 0}$. More generally, given $A_{\bullet} \in \CAlgTow$ and $X$ a finite space, the tower $X \otimes A_{\bullet} \in \CAlgTow$ gives a filtration on $X \otimes A_{0}$, which is also called the \emph{May filtration} in \cite[3.3.3]{angelini-knoll-salch}. 

The remainder of this subsection is devoted an \categorical treatment of results in \cite{angelini-knoll-salch}. In particular, Proposition \ref{prop:THHMayEv0} allows us to define the May filtration on $\THH$ in the \categorical language.

\begin{prop}\label{prop:THHMayEv0}
	Let $A \in \Alg_{\geq 0}$ and let $A_{\bullet}$ denote the Whitehead tower of $A$. There is a canonical equivalence of spectra with $\mathbb{T}$-action, $\ev_{0}\THHDay(A_{\bullet}) \simeq \THH(A)$. In the case where $A \in \CAlg_{\geq 0 }$, this is an equivalence in $\CAlg^{B\mathbb{T}}$.
\end{prop}

\begin{proof}
	Proposition \ref{prop:Daycofinal} states that $\ev_{0} : \tow(\Sp) \rightarrow \Sp$ is symmetric monoidal so we may apply Lemma \ref{lem:Env} to extract the following commutative square.
	\begin{center}
	\begin{tikzcd}
	\tow(\Sp)^{\otimes}_{\Act} \arrow[r,"\Day"] \arrow[d,"(\ev_{0})^{\otimes}_{\emph{act}}"'] & \tow(\Sp) \arrow[d,"\ev_{0}"] \\
	\Sp_{\Act}^{\otimes} \arrow[r,"\otimes"'] & \Sp 
	\end{tikzcd}		
	\end{center}

	\noindent Additionally, note that $\ev_{0}$ commutes with colimits and that $\ev_{0}A_{\bullet} \simeq A$. We obtain the following chain of canonical equivalences in $\Sp^{B\mathbb{T}}$:
	\begin{align*}
	\ev_{0}\THHDay(A_{\bullet})& = \ev_{0}\varinjlim\left(\Day \circ (A^{\otimes}_{\bullet})_{\Act} \circ V^{o} \right) \\
							   & \simeq \varinjlim\left( \ev_{0} \circ \Day \circ (A^{\otimes}_{\bullet})_{\Act} \circ V^{o}\right) \\
							   & \simeq \varinjlim\left(\otimes \circ A^{\otimes}_{\Act} \circ V^{o} \right) \\
							   & = \THH(A). 	
	\end{align*}
	In the case where $A \in \CAlg_{\geq 0}$, a similar argument works, since the induced functor $\ev_{0} : \CAlg^{\tow} \rightarrow \CAlg$ preserves sifted colimits, by \cite[3.2.3.2]{HA}.	
\end{proof}

\begin{rmk}
	The proof of Proposition \ref{prop:THHMayEv0} yields a slightly more general result; there is a canonical equivalence of spectra with $\mathbb{T}$-action, $\ev_{0} \THHDay(A_{\bullet}) \simeq \THH(A_{0})$, for any $A_{\bullet} \in \Alg^{\tow}$.	
\end{rmk}

\begin{defn}\cite[3.3.3, 3.4.9]{angelini-knoll-salch}
	Let $A \in \Alg_{\geq 0}$ and let $A_{\bullet}$ denote the Whitehead tower of $A$. The \emph{May filtration} of $\THH(A)$ is the tower $\THHDay(A_{\bullet})$. We will let $\Fil_{i}^{\May}$ denote $\ev_{i}\THHDay(A_{\bullet})$ and let $\gr_{i}^{\May}$ denote $\Fil_{i}^{\May}/\Fil_{i+1}^{\May}$. While our notation makes no reference to $A$, the context of our usage should indicate the underlying connective ring spectrum. 
	
	If we view $\THH$ as a functor, $\THH : \Alg \rightarrow \Sp^{B\mathbb{T}}$, we obtain a tower of functors
	$$\cdots \rightarrow \Fil^{\May}_{i} \rightarrow \cdots \rightarrow \Fil^{\May}_{1} \rightarrow \THH.$$
	\noindent We call this tower the \emph{May filtration of the functor} $\THH$, and by taking cofibers, it induces another tower 
	$$ \cdots \rightarrow \THH/\Fil^{\May}_{i} \rightarrow \cdots \rightarrow  \THH/\Fil^{\May}_{1} \rightarrow 0.$$
\end{defn}

The following is the \categorical version of \cite[3.3.10]{angelini-knoll-salch}, which is crucial for the sequel. 

\begin{prop}\label{prop:DayGraded}
	For $A_{\bullet} \in \Alg^{\tow}$, there is a canonical equivalence of spectra with $\mathbb{T}$-action, $\gr_{\ast}\THHDay(A_{\bullet}) \simeq \THH(\gr_{\ast}A_{\bullet})$. In the case where $A_{\bullet} \in \CAlg^{\tow}$, this is an equivalence in $\CAlg^{B\mathbb{T}}$.
\end{prop}

\begin{proof}
	By Proposition \ref{prop:grmonoidal} and Lemma \ref{lem:Env}, it follows that
	\begin{center}
	\begin{tikzcd}
		\tow(\Sp)^{\otimes}_{\Act} \arrow[r,"\Day"] \arrow[d,"(\gr_{\ast}^{\otimes})_{\Act}"'] & \tow(\Sp) \arrow[d,"\gr_{\ast}"] \\
		\Sp_{\Act}^{\otimes} \arrow[r,"\otimes"'] & \Sp
	\end{tikzcd}		
	\end{center}

	\noindent commutes. Additionally, since $\gr_{\ast}$ preserves, 
	
	\begin{align*}
	\gr_{\ast}\THHDay(A_{\bullet}) & = \gr_{\ast}\varinjlim\left(\Day \circ (A^{\otimes}_{\bullet})_{\Act} \circ V^{o} \right) \\
							   & \simeq \varinjlim\left( \gr_{\ast} \circ \Day \circ (A^{\otimes}_{\bullet})_{\Act} \circ V^{o}\right) \\
							   & \simeq \varinjlim\left(\otimes \circ (\gr_{\ast}A_{\bullet})^{\otimes}_{\Act} \circ V^{o} \right) \\
							   & = \THH(\gr_{\ast}A_{\bullet}).
	\end{align*}
	In the case where $A \in \CAlg_{\geq 0}$, a similar argument works, since the induced functor $\gr_{\ast} : \CAlg^{\tow} \rightarrow \CAlg$ preserves sifted colimits, by \cite[3.2.3.2]{HA}.	
\end{proof}

An immediate consequence of the previous result is the following:

\begin{cor}\emph{\cite[3.4.11]{angelini-knoll-salch}}\label{cor:THHMayGraded}
	For $A \in \CAlg_{\geq 0}$ and $A_{\bullet}$ the Whitehead tower of $A$, there is a canonical equivalence in $\CAlg^{B\mathbb{T}}$, $\gr^{\May}_{\ast}\THH(A) \simeq \THH(H\pi_{\ast}A)$. 	
\end{cor}

\noindent This result, combined with several others, is used to construct the $\THH$-May spectral sequence sequence $E_{\ast,\ast}^{1}\cong \pi_{\ast}\THH(H\pi_{\ast}A) \Rightarrow \pi_{\ast}\THH(A)$; see \cite[3.4.8]{angelini-knoll-salch}.

\begin{prop}\emph{\cite[3.5.4]{angelini-knoll-salch}}\label{prop:THHMayConnectivity}
	If $A_{\bullet} \in \Alg^{\tow}_{\geq 0}$ with the property that $\pi_{m}A_{n} \cong 0$ for all $m<n$, then $\pi_{m}\Fil_{n}^{\May} \cong 0$ for all $m<n$ as well. 	
\end{prop}

\begin{proof}
	First, recall that if $E$ and $F$ are $n$ and $m$-connective spectra, respectively, then $E \otimes F$ is $(n+m)$-connective. Next, note that $\Fil_{n}^{\May}$ is the geometric realization of a simplicial object whose $k$-th term is given by 
	$$ \ev_{n}(A_{\bullet}^{\Day k}) = \underset{i_{1} + \cdots + i_{k} \geq n}{\varinjlim} A_{i_{1}} \otimes \cdots \otimes A_{i_{k}}.$$ 
	\par \noindent As each $A_{i_{m}}$ is $i_{m}$-connective, by our recollection, each term in the simplicial object is $n$-connective, and hence $\Fil_{n}^{\May}$ is $n$-connective.
\end{proof}

\begin{cor}\label{cor:THHMayHausdorff}
	For $A \in \Alg_{\geq 0}$, the May filtration has the property that
	$$\underset{n \geq 0}{\varprojlim}\left(\Fil_{n}^{\May}\right) \simeq 0;$$
	\noindent i.e. the filtration is Hausdorff. Additionally, the May filtration on the functor $\THH$ is also Hausdorff. 
\end{cor}

\begin{proof}
	By Proposition \ref{prop:THHMayConnectivity} and a Bousfield-Kan spectral sequence calculation, $\Fil^{\May}_{n}$ is $n$-connective. Since the connectivity grows linearly in the tower, we conclude $\underset{n \geq 0}{\varprojlim}\left(\Fil_{n}^{\May}\right) \simeq 0$. 	
\end{proof}

\begin{cor}\label{cor:THHMayConvergence}
	Let $A \in \Alg_{\geq 0}$. Then the tower 	
	$$ \cdots \rightarrow \THH(A)/\Fil^{\May}_{i} \rightarrow \cdots \rightarrow \THH(A)/\Fil^{\May}_{2} \rightarrow \THH(A)/\Fil^{\May}_{1}$$
	
	\noindent converges to $\THH(A)$.
\end{cor}

\begin{proof}
	For $n \geq 1$ there are fiber sequences $\Fil^{\May}_{i} \rightarrow \THH(A) \rightarrow \THH(A)/\Fil^{\May}_{i}$. Passing to the limit and applying Corollary \ref{cor:THHMayHausdorff}, we conclude.	
\end{proof}

\begin{rmk}
	As limits are calculated pointwise in functor categories, the lemma above implies that the tower of functors $\{\THH/\Fil^{\May}_{i}\}_{i \geq 1}$ converges to $\THH$.
\end{rmk}

\begin{rmk}
	While working with $\tow(\Sp)$ suffices for our purposes, by Remark \ref{rmk:GeneralTHH}, we can define $\THH^{\Fil}(A_{\bullet})$ for $A_{\bullet} \in \Alg^{\Fil}$, satisfying many of the same properties. For example, analogues of Propositions \ref{prop:THHMayEv0} and \ref{prop:DayGraded} hold for $\THH^{\Fil}$; instead of using the symmetric monoidality of $\ev_{0}$, we use the symmetric monoidality of $\varinjlim : \Fil(\Sp) \rightarrow \Sp$, and instead of the associated graded functor $\gr_{\ast} : \tow(\Sp) \rightarrow \Sp$, we use the functor $\gr_{\ast}^{\Fil}$ from Section \ref{section:gr}. 
\end{rmk}

\subsection{The May filtration on $\THH_{hC_{p}}$} 
In order to prove our main theorem, we need to establish variants of Proposition \ref{prop:THHMayEv0} and Corollaries \ref{cor:THHMayGraded}, \ref{cor:THHMayHausdorff}, and \ref{cor:THHMayConvergence} for $\THH(-)_{hC_{p}}$. We proceed by applying $C_{p}$-homotopy fixed points to the tower $\THHDay(A)$, and mimicking the proofs in the previous subsection.

\begin{prop}
	Let $A \in \Alg_{\geq 0}$, and let $A_{\bullet}$ denote the Whitehead tower of $A$. There is an equivalence of spectra with $\mathbb{T}$-action, $\ev_{0}\left(\THHDay(A_{\bullet})_{hC_{p}}\right) \simeq \THH(A)_{hC_{p}}$. In the case where $A \in \CAlg_{\geq 0}$, this is an equivalence in $\CAlg^{B\mathbb{T}}$.
\end{prop}

\begin{proof}
	Since the functor $\ev_{0}$ is colimit-preserving, we have:
	\begin{align*}
	\ev_{0}\left(\THHDay(A_{\bullet})_{hC_{p}}\right) & \simeq \left(\ev_{0}\THHDay(A_{\bullet})\right)_{hC_{p}} \\
												     & \simeq \THH(A)_{hC_{p}}.
	\end{align*}
	\noindent As in Proposition \ref{prop:THHMayEv0}, similar considerations yield the commutative variant.	
\end{proof}

\begin{defn}
	Let $A \in \Alg_{\geq 0}$ and let $A_{\bullet}$ denote the Whitehead tower of $A$. We define the \emph{May filtration} on $\THH(A)_{hC_{p}}$ as $\THHDay(A_{\bullet})_{hC_{p}}$. We will denote the $i$-th piece of this filtration by $\Fil^{\May}_{i}(p)$, and the $i$-th graded piece by $\gr^{\May}_{i}\THHDay(A_{\bullet})_{hC_{p}}$.
\end{defn}

\begin{prop}
	Let $A \in \Alg_{\geq 0}$, and let $A_{\bullet}$ denote the Whitehead tower of $A$. There is an equivalence of spectra with $\mathbb{T}$-action, $\gr_{\ast}^{\May}\left(\THHDay(A_{\bullet})_{hC_{p}}\right) \simeq \THH(H\pi_{\ast}A)_{hC_{p}}$. In the case where $A \in \CAlg_{\geq 0}$, this is an equivalence in $\CAlg^{B\mathbb{T}}$.	
\end{prop}

\begin{proof}
	Since the functor $\gr_{\ast}$ is colimit-preserving, we have:
	\begin{align*}
	\gr_{\ast}\left(\THHDay(A_{\bullet})_{hC_{p}}\right) & \simeq \left(\gr_{\ast}\THHDay(A_{\bullet})\right)_{hC_{p}} \\
	& \simeq \THH(H\pi_{\ast}A)_{hC_{p}}.
	\end{align*}
	\noindent As in Proposition \ref{prop:DayGraded}, similar considerations yield the commutative variant.		
\end{proof}

\begin{prop}\label{prop:THHDayCpConnectivity}
	If $A_{\bullet} \in \CAlg_{\geq 0}^{\tow}$ with the property that $\pi_{m}A_{n} \cong 0$ for all $m < n$, then $\pi_{m}\Fil^{\May}_{n}(p) \cong 0$, for all $m < n$ as well. 	
\end{prop}

\begin{proof}
	It is standard that homotopy orbits preserve connectivity, so we may conclude by Proposition \ref{prop:THHMayConnectivity}.
\end{proof}

\begin{cor}
	For $A \in \CAlg_{\geq 0}$, the May filtration of $\THH(A)_{hC_{p}}$ is Hausdorff. 	
\end{cor}

\begin{proof}
	By Proposition \ref{prop:THHDayCpConnectivity}, the connectivity of the tower grows linearly, so
	$$\underset{n \geq 0}{\varprojlim}\Fil^{\May}_{n}(p) \simeq 0.$$
\end{proof}

\section{Proof of Theorem \ref{thm:maintheorem}}

In this section we prove our main result. Using Proposition \ref{prop:SpecialCase}, we deduce the graded pieces of the May filtration are fpqc sheaves and conclude the result for $\THH$ by induction up the tower in Corollary \ref{cor:THHMayConvergence}. By similar methods, we prove the claim for $\THH_{hC_{p}}$, and  use the structure of $\CycSp$ to prove it for $\TC$ as well.

\begin{lem}\label{lem:grCech}
	Let $f : A \rightarrow B$ be a faithfully flat morphism in $\CAlg_{\geq 0}$. Let $f_{\bullet} : A_{\bullet} \rightarrow B_{\bullet}$ denote the induced map between their Whitehead towers and let $f_{\ast} : H\pi_{\ast}A \rightarrow H\pi_{\ast}B$ denote $\gr_{\ast}(f_{\bullet})$. There is a canonical equivalence of augmented cosimplicial commutative ring spectra $\gr_{\ast}\left( C(f_{\bullet})_{+} \right) \simeq C(f_{\ast})_{+}$. 
\end{lem}

\begin{proof}
	As \Cech conerves are computed by left Kan extensions, there is a natural transformation $C^{\bullet}(f_{\ast})_{+} \rightarrow \gr_{\ast} C^{\bullet}(f_{\bullet})_{+}$ since $\gr_{\ast} C^{\bullet}(f_{\bullet})_{+}|_{\Delta^{\leq 0}_{+}} = f_{\ast}$. For $n > 0$, we have a canonical map
	$$C^{n}(f_{\ast})_{+} \simeq H\pi_{\ast}B \otimes_{H\pi_{\ast}A} \cdots \otimes_{H\pi_{\ast}A} H\pi_{\ast}B \rightarrow H\pi_{\ast}\left(B \otimes_{A} \cdots \otimes_{A} B\right) \simeq  \gr_{\ast}C^{n}(f)_{+}.$$
	\noindent Since $A \rightarrow B$ is faithfully flat, the usual K{\"u}nneth spectral sequence calculation shows that this map is an equivalence. Thus, the morphism of \Cech conerves is an equivalence. 	
\end{proof}

\begin{cor}
	The functors $\gr^{\May}_{\ast}\THHDay, \gr^{\May}_{\ast}\THHDay(-)_{hC_{p}} : \CAlg_{\geq 0} \rightarrow \Sp$ are sheaves for the fpqc topology. 
\end{cor}

\begin{proof}
	By Proposition \ref{prop:DayGraded}, there is an equivalence of functors $\gr^{\May}_{\ast}\THH \simeq \THH(\gr_{\ast}(-))$. Now let $f: A \rightarrow B$ be a faithfully flat morphism in $\CAlg_{\geq 0}$ and observe that by Lemma \ref{lem:grCech}, $C^{\bullet}(f_{\ast}) \simeq \gr_{\ast}C^{\bullet}(f)$. Thus, Proposition \ref{prop:SpecialCase} applied to the commutative diagram
	\begin{center}
	\begin{tikzcd}
	\gr_{\ast}^{\May}\THH(A) \arrow[r] \arrow[d] & \underset{\Delta}{\varprojlim}\gr^{\May}_{\ast} \THH(C^{\bullet}(f)) \arrow[d] \\
	\THH(H\pi_{\ast}A) \arrow[r] & \underset{\Delta}{\varprojlim} \THH(C^{\bullet}(f_{\ast}))
	\end{tikzcd}		
	\end{center}

	\noindent proves the claim for $\gr^{\May}_{\ast}\THHDay$. An identical proof works for $\gr^{\May}_{\ast}\THHDay(-)_{hC_{p}}$ as well. 
\end{proof}

\begin{prop}\label{prop:grMaySheaves}
	For all $i \geq 0$, the functors $\gr^{\May}_{i} \THH(-)$ are fpqc sheaves.	
\end{prop}

\begin{proof}	
	By definition, $\gr^{\May}_{\ast}\THH \simeq \bigoplus_{i \geq 0} \gr^{\May}_{i} \THH$, so that $\gr_{i}^{\May}\THH$ is a retract of $\gr^{\May}_{\ast}\THH$. This implies that $\gr_{i}^{\May} \THH$ is an fpqc sheaf for all $i \geq 0$. 
\end{proof}

\begin{cor}
	For all $i \geq 1$, the functors $\THH/\Fil_{i}^{\May} : \CAlg_{\geq 0} \rightarrow \Sp$ are fpqc sheaves.	
\end{cor}

\begin{proof}
	Induct on $i$ using Proposition \ref{prop:grMaySheaves} and the fiber sequences 
	$$\gr_{i}^{\May}\THH \rightarrow \THH/\Fil^{\May}_{i+1} \rightarrow \THH/\Fil^{\May}_{i}.$$
\end{proof}

\begin{thm}
	The functors $\THH, \THH_{hC_{p}} : \CAlg_{\geq 0} \rightarrow \Sp$ are sheaves for the fpqc topology.	
\end{thm}

\begin{proof}
	Since limits of $\Sp$-valued sheaves can be computed in the category of $\Sp$-valued presheaves, the equivalence
	$$\THH \simeq \underset{i \geq 1}{\varprojlim} \left(\THH/\Fil^{\May}_{i}\right),$$
	\noindent guarantees $\THH$ is an fpqc sheaf. Similarly, since 
	$$\THH(-)_{hC_{p}} \simeq \underset{i \geq 1}{\varprojlim}\left(\THH(-)_{hC_{p}}/\Fil^{\May}_{i}(p)\right),$$
	\noindent  $\THH(-)_{hC_{p}}$ is an fpqc sheaf as well.
\end{proof}

\begin{cor}
	The functors $\THH(-)^{h\mathbb{T}}$, $\THH(-)^{hC_{p}}$, and $\THH(-)^{tC_{p}}$ are fpqc sheaves on $\CAlg_{\geq 0}$. 	
\end{cor}

\begin{proof}
	The first two functors are fpqc sheaves because homtopy fixed points are calculated via limits. The third is an fpqc sheaf because there is a cofiber sequence 
	$$ \THH(-)_{hC_{p}} \rightarrow \THH(-)^{hC_{p}} \rightarrow \THH(-)^{tC_{p}}$$
	\noindent and $\THH(-)_{hC_{p}}$ and $\THH(-)^{hC_{p}}$ are fpqc sheaves.
\end{proof}

Using the results above, we will show that $\THH : \CAlg_{\geq 0} \rightarrow \CycSp$ is an fpqc sheaf. As a Corollary, we will show $\TC$ is as well. We now recall the construction of $\CycSp$ and the definition of topological cyclic homology. 

\begin{defn}\cite[II.1.6]{nikolaus-scholze}
	Let $\mathbb{P}$ denote the set of all primes. The \category of \emph{cyclotomic spectra} is the pullback
	\begin{center}
	\begin{tikzcd}
		\CycSp \arrow[r] \arrow[d] &   \Fun(\Delta^{1},\prod_{p\in \mathbb{P}} \Sp^{B\mathbb{T}}) \arrow[d, "(\ev_{0} ; \ev_{1})" ] \\
		\Sp^{B\mathbb{T}} \arrow[r,"F"'] &  \prod_{p} \Sp^{B\mathbb{T}} \times \prod_{p\in \mathbb{P}} \Sp^{B\mathbb{T}}
	\end{tikzcd}	
	\end{center}
	\noindent where the functor $F$ is given by the formula $F = \left(\prod_{p\in \mathbb{P}} id, \prod_{p \in \mathbb{P}} (-)^{tC_{p}}\right)$. By \cite[II.1.5]{nikolaus-scholze}, $\CycSp$ is stable and presentable, hence enriched over $\Sp$. The \emph{topological cyclic homology} of a cyclotomic spectrum, $X$, is the mapping spectrum
	$$\TC(X) = \Map_{\CycSp}(\mathbb{S}^{\text{triv}}, X),$$
	
	\noindent where $\mathbb{S}^{\text{triv}}$ is the sphere spectrum with the trivial $\mathbb{T}$-action and Frobenius maps given by the composite $\mathbb{T}$-equivariant maps $\sphere \rightarrow \sphere^{hC_{p}} \rightarrow \sphere^{tC_{p}}$; see \cite[I.2.3.i, II.1.2.ii]{nikolaus-scholze} for further details. As $\THH(A)$ admits the structure of a cyclotomic spectrum for all $A \in \Alg$ (\cite[III.2]{nikolaus-scholze}), we define $\TC(A) = \TC(\THH(A))$. 	
\end{defn}

\begin{prop}
	$\THH : \CAlg_{\geq 0} \rightarrow \CycSp$ is a fpqc sheaf. 	
\end{prop}

\begin{proof}
	Because $\CycSp$ is defined as as lax equalizer, which is a strict pullback of $\infty$-categories, it suffices to check that the functors obtained from $\THH$ by post-composing with the projections are all fpqc sheaves. Indeed, we obtain the following three functors:
	
	\begin{enumerate}
		\item $\THH: \CAlg_{\geq 0} \rightarrow \Sp^{B\mathbb{T}}$;
		\item $\prod_{p} (\varphi_{p}(-)) : \CAlg_{\geq 0} \rightarrow \prod_{p} \Fun(\Delta^{1},\Sp^{B\mathbb{T}})$; and by identifying composites,
		\item $ \left(\prod_{p}\THH(-), \prod_{p} \THH(-)^{tC_{p}}\right) : \CAlg_{\geq 0} \rightarrow \prod_{p} \Sp^{B\mathbb{T}} \times \prod_{p} \Sp^{B\mathbb{T}}$.		
	\end{enumerate}
	
	\noindent The functor in (1) is an fpqc sheaf because $\THH : \CAlg_{\geq 0} \rightarrow \Sp$ is and limits in $\Sp^{B\mathbb{T}}$ are calculated pointwise. Note that the functor in (2) is a sheaf if the functor in (3) is as well. To show the functor in (3) is an fpqc sheaf, we can project onto the different factors and reduce to checking that $\THH$ and $\THH^{tC_{p}}$ are fpqc sheaves on $\CAlg_{\geq 0}$. Thus, the claim is proved.
\end{proof}

\begin{cor}
	$\TC : \CAlg_{\geq 0} \rightarrow \Sp$ is an fpqc sheaf.	
\end{cor}

\begin{proof}
	As $\TC(A) = \Map_{\CycSp}(\mathbb{S}^{\text{triv}},\THH(A))$, and $\THH$ is an fpqc sheaf valued in $\CycSp$, for $A \rightarrow B$ faithfully flat, we have an equivalence
	$$ \TC(A) \rightarrow  \Map_{\CycSp}(\mathbb{S}^{\text{triv}},\varprojlim \THH(B^{\otimes_{A} \bullet})) \simeq \varprojlim \Map_{\CycSp}(\mathbb{S}^{\text{triv}},\THH(B^{\otimes_{A} \bullet})) = \varprojlim \TC(B^{\otimes_{A} \bullet})$$
\end{proof}

\bibliography{thh_fpqc_descent_references}{}
\bibliographystyle{unsrt}

\end{document}